\documentclass[11pt,leqno]{amsart}
\usepackage{amsthm,amsfonts,amssymb,amsmath,oldgerm,enumitem}
\usepackage{epsfig}
\numberwithin{equation}{section}
\usepackage[thinlines]{easybmat}

\newcommand\tRR{\widetilde{\mathbb R}}
\renewcommand\d{\partial}

\newcommand\R{\mathbb R}
\newcommand\C{\mathbb C}
\newcommand\Z{\mathbb Z}
\newcommand\N{\mathbb N}
\def\eps{\varepsilon}

\newcommand\br{\begin{remark}}
\newcommand\er{\end{remark}}
\newcommand\bp{\begin{pmatrix}}
\newcommand\ep{\end{pmatrix}}
\newcommand\be{\begin{equation}}
\newcommand\ee{\end{equation}}
\newcommand\ba{\begin{equation}\begin{aligned}}
\newcommand\ea{\end{aligned}\end{equation}}

\newcommand{\bap}{\begin{app}}
\newcommand{\eap}{\end{app}}
\newcommand{\begs}{\begin{exams}}
\newcommand{\eegs}{\end{exams}}
\newcommand{\beg}{\begin{example}}
\newcommand{\eeg}{\end{exaplem}}
\newcommand{\bpr}{\begin{proposition}}
\newcommand{\epr}{\end{proposition}}
\newcommand{\bt}{\begin{theorem}}
\newcommand{\et}{\end{theorem}}
\newcommand{\bc}{\begin{corollary}}
\newcommand{\ec}{\end{corollary}}
\newcommand{\bl}{\begin{lemma}}
\newcommand{\el}{\end{lemma}}
\newcommand{\bd}{\begin{definition}}
\newcommand{\ed}{\end{definition}}
\newcommand{\brs}{\begin{remarks}}
\newcommand{\ers}{\end{remarks}}

\newtheorem{theo}{Theorem}[section]

\newtheorem{exams}[theo]{Examples}

\numberwithin{equation}{section}

\newcommand{\Id}{{\rm Id }}

\newcommand{\diag}{{\rm diag }}

\newtheorem{theorem}{Theorem}[section]
\newtheorem{proposition}[theorem]{Proposition}
\newtheorem{corollary}[theorem]{Corollary}
\newtheorem{lemma}[theorem]{Lemma}
\newtheorem{definition}[theorem]{Definition}

\newtheorem{example}[theorem]{Example}
\newtheorem{remark}[theorem]{Remark}

\newcommand\cE{{\mathcal  E}}

\newcommand\cO{{\mathcal O}}

\newcommand\cM{{\mathcal M}}

\newcommand{\lot}{\ell.o.t.}
\DeclareMathOperator{\dD}{d}
\DeclareMathOperator\eD{e}
\DeclareMathOperator\iD{i}
\DeclareMathOperator\sign{sign}
\DeclareMathOperator\Tr{Tr}

\title
[Linear damping estimates for inviscid roll waves]
{Linear damping estimates for periodic roll wave solutions of 
the inviscid Saint-Venant equations and related systems of hyperbolic balance laws}

\author{L.~Miguel Rodrigues}
\address{Univ Rennes, CNRS, IRMAR - UMR 6625, F-35000 Rennes, France}
\email{luis-miguel.rodrigues@univ-rennes1.fr}
\thanks{Research of L.M.R. was partially supported by the ANR Project HEAD ANR-24-CE40-3260 and the Institut Universitaire de France.}

\author{Kevin Zumbrun}
\address{Indiana University, Bloomington, IN 47405}
\email{kzumbrun@indiana.edu}
\thanks{Research of K.Z. was partially supported
under NSF grants no. DMS-1400555 and DMS-1700279}

\begin{document}

\begin{abstract}
	Substantially extending previous results of the authors for smooth solutions in the viscous case, 
	we develop linear damping estimates for periodic roll-wave solutions 
	of the inviscid Saint-Venant equations and related systems of hyperbolic balance laws.
	Such damping estimates, consisting of $H^s$ energy estimates yielding exponential slaving of 
	high-derivative to low-derivative norms, have served as crucial ingredients in nonlinear stability
	analyses of traveling waves in hyperbolic or partially parabolic systems, both in obtaining
	high-frequency resolvent estimates and in closing a nonlinear iteration for which available linearized
	stability estimates apparently lose regularity.
	Here, we establish for systems of size $n\leq 6$ a Lyapunov-type theorem stating that such energy 
	estimates are available whenever strict high-frequency spectral stability holds;
	for dimensions $7$ and higher, there may be in general a gap between high-frequency spectral 
	stability and existence of the type of energy estimate that we develop here.
	A key ingredient is a dimension-dependent linear algebraic lemma reminiscent of 
	Lyapunov's Lemma for ODE that is to our knowledge new.

\vspace{0.5em}


\vspace{0.5em}

\end{abstract}

\maketitle


\setcounter{tocdepth}{1}
\tableofcontents

\section{Introduction}\label{s:intro}

In the present contribution, we construct linear damping estimates for roll wave solutions of the Saint-Venant equations (SV) for inclined shallow-water flow, or, more generally, discontinuous periodic traveling waves of a general system of hyperbolic balance laws
\be\label{blaw}
\d_t(f_0(W))+ \d_x(f(W))=R(W), \qquad W\in \R^n.
\ee

Roll waves are potentially harmful periodic wave trains forming in canals or other channel flow under situations of ``hydrodynamic instability,'' that is, when constant-height laminar flow becomes unstable.  Oscillating in amplitude, roll waves can have substantially larger maximum fluid height than a corresponding laminar flow carrying the same flow, leading to spillage or possible damage. This, along with their dramatic nature, has led to considerable interest in their existence and stability, or ``permanence'' as persistent waves. See, e.g., \cite{Je,Dr,Br1,Br2} for further discussion.

Stability of smooth roll wave solutions of the  ``viscous'' version of the Saint-Venant equations has been in principle completely resolved in \cite{JZN,JNRZ13,RZ,BJNRZ2}, with, on one hand, the identification of spectral conditions under which one obtains both nonlinear stability and a detailed description of asymptotic behavior, and on the other hand the study of those spectral conditions by a combination of thorough numerical computations to derive complete stability diagrams and rigorous near-onset asymptotic analyses in regimes not accessible by numerics.

However, as noted in \cite{BJNRZ2,JNRYZ}, there is considerable advantage to working with the inviscid equations in understanding behavior. In particular, a useful power-law description of stability obtained in \cite{BJNRZ2} for the regime relevant to hydraulic engineering corresponds effectively to an inviscid limit, requiring intensive computational resources to resolve in the viscous setting. By contrast, in the inviscid spectral stability analysis of \cite{JNRYZ}, the corresponding stability boundaries could be found with orders of magnitude fewer computations; in particular, the low-frequency stability boundary was obtained explicitly, as the solution of a cubic equation in the model parameters. And, indeed, it is the inviscid equations that appear to be the industry standard in hydraulic engineering.

These considerations motivate the study of linearized and nonlinear stability in the original inviscid, quasilinear hyperbolic form \eqref{blaw}, despite the technical difficulties, among others, of {discontinuity of the background wave}, {lack of parabolic smoothing} and {presence of characteristic points} in the equations. Such analyses have been carried out in the scalar case in \cite{DR1,DR2,GR} (for which however periodic solutions are always unstable) and for front-type solutions of some general systems including the Saint-Venant equations in \cite{YZ,FR1,FRYZ,FR2}. As a bridge between the viscous and inviscid worlds we point out that the asymptotic stability result of \cite{Blochas-Rodrigues} are uniform with respect to viscosity.

Here, generalizing analysis of \cite{RZ} in the viscous case, we develop for periodic solutions of (SV) and related systems \eqref{blaw}, under the assumption of strict high-frequency spectral stability, a Lyapunov-type {\it linear damping estimate} yielding exponential slaving of higher- to lower-derivative Sobolev norms.

Such estimates, in a nonlinear form, are a key ingredient in the study of nonlinear stability  in situations of delicate regularity, compensating in a nonlinear iteration scheme for apparent derivative loss in linearized estimates used to obtain decay.  See \cite{Z,MZ_large,RZ} and \cite{YZ,FR1,FR2} for examples from 
hyperbolic-parabolic systems and quasilinear hyperbolic systems. Here, we carry out the first, important step of identifying an underlying structure/mechanism by which spectral properties may be translated into a corresponding linear energy estimate. For systems of dimension $n\leq 6$, these estimates are seen to be sharp in the sense that they may be obtained whenever there holds the (necessary) condition of high-frequency spectral stability. For systems of higher dimension $n\geq 7$, we conjecture but have not established definitely that there is a gap between the conditions for high-frequency spectral stability and existence of a linear energy estimate of the type we construct here. For the artificial complex analogs of these conditions, we prove on one hand that no gap exists when $n\leq 4$ whereas by explicit example we show that such a gap does exist in dimensions $n\geq 5$. 

As analyzed in details in \cite[Appendix~A]{FR1} we stress that if one were willing to derive a full damping estimate, that is a damping of the whole solution, not only its high-frequency part then a similar gap appears already for the stability of \emph{constant} solutions when the system contains $n=3$ equations, thus essentially always. Hence in the wave-stability theory the need to combine high-frequency damping estimates with more advanced functional-analytic techniques.

Along the way, we encounter an interesting linear algebraic fact, Lemma \ref{linalg_lem} analogous to those underlying the predecessors of our analysis by Lyapunov \cite{Ly} and Kreiss \cite{Kreiss}, but confined to spaces of dimensions three and lower for complex matrices and five and lower for real ones, that is the key observation needed for the extension to systems \eqref{blaw} of size $n>2$. This is in turn closely related to the algebraic geometry question ``what is the largest number of quadratic equations $q_j(x)=0$, $x\in \C^2$, with associated quadratic forms $Q_j$, for which it is guaranteed that either (i) the family $q_j(x)=0$ has a common nontrivial solution, or (ii) there is a linear combination $\sum_j c_j Q_j$ that is positive definite?'' The answer, ``two,'' leads to the bounds of three (complex case) and five (real case) in Lemma \ref{linalg_lem}, the latter in turn leading to the bound $n\leq 6$ for system \eqref{blaw}.

The above observations lead to sharp high-frequency treatment of roll waves in systems \eqref{blaw} of dimension 
$6$ and lower, notably the Saint-Venant equations (SV) and the ($3\times 3$ version of the) Richard-Gavrilyuk model (RG) recently introduced in \cite{RG1} as a refinement incorporating effects of turbulent vorticity. For systems in dimension seven and higher, our methods are not guaranteed to work whenever high-frequency stability holds, but, due to the above-mentioned theoretical gap, only under the stronger condition for existence of an energy estimate. In practice, however, our stronger condition may well be sufficient, as this gap appears to occur rather infrequently among randomly chosen systems, and when it does occur is not large; see Remark~\ref{gaprmk}.

\subsection{Reader's guide}\label{s:overview}
The underlying principle for our analysis, originating in \cite{Z,MZ_large} and greatly extended in \cite{RZ,FR1,FR2} and elsewhere, is that strict high-frequency spectral estimates should be related to high-frequency damping, through resolvent estimates obtained by the same (WKB-type or other) estimates. That is, high-frequency damping is a restricted, high-frequency version of the type of energy estimates obtained by Lyapunov \cite{Ly} for initial value problems in ODE and Kreiss \cite{Kreiss} for initial boundary value problems in hyperbolic PDE. So, what we are really trying to demonstrate is a Kreiss symmetrizer type estimate (or Lyapunov lemma, in case of ODE), showing that such an estimate may be obtained equally by energy methods. The advantage of the energy estimate formulation, of course, is {robustness} under perturbation, in particular extension to the nonlinear setting. In the particular case of damping, this gives a crucial control of regularity as well.

In the present setting, there are some additional complications due to the presence of shock discontinuities, but also some simplifications due to the restriction to high frequencies. To aid the reader, we will try to isolate the main ideas and difficulties here in an informal way, sweeping aside some of the technicalities in the analysis of Section \ref{s:sv}.

Let us first describe the analytical setting. A roll wave by our definition consists of a piecewise smooth periodic traveling wave, with smooth portions separated by shock discontinuities. By considerations of well-posedness, these discontinuities must be of admissible Lax type, with $n+1$ entering characteristics and $n-1$ exiting characteristics from the shock, and Rankine-Hugoniot jump conditions solvable for exiting modes in terms of entering ones.

At the nonlinear level, there are additional complications due to movement of the shock, with modulating
phase shifts $\psi(x,t)$ in $x$ introduced to fix the shocks at their initial location; this is discussed further in Section \ref{s:sv}. However, at the linear level this may more or less be ignored, as phase shifts may be eliminated from the jump conditions, then recovered later. Thus, the reader will lose little by ignoring the discussion of phase shift in Section~\ref{s:sv}.

{\bf Sonic points.} A more immediate complication is the appearance of ``sonic'' or ``characteristic'' points, 
where the hyperbolic characteristic speed is equal to the speed of the traveling wave. Assuming strict hyperbolicity of \eqref{blaw}, as holds for the fluid-dynamical examples we have in mind, hyperbolic characteristic speeds relative to the shock are real, changing sign only at sonic points. But, characteristics entering the shock from the left are the ones of positive relative speed, as are the characteristics exiting the shock on the right, and by the Lax characteristic condition these are of different number. Thus, periodicity implies passage through a sonic point/change of sign of at least one relative characteristic speed, as $x$ passes from the left to the right boundary of one smooth periodic cell.

The effect of a sonic point on the eigenvalue ODE $\lambda A_0 w + (Aw)'= Ew$ is that the principal part $Aw'$ becomes singular, making it a singular ODE problem. See for example the discussions in \cite{N2,JNRYZ, DR2} in the scalar and Saint-Venant setting, and some initial consequences for the spectral problem. A definitive study on the effects of characteristic points has been made in the scalar context in \cite{DR2,GR}, showing also the consequences for resolvent and energy estimates, and we make important use of those ideas here.

{\bf Sonic vs. transverse modes.} The main idea behind the WKB-type high-frequency spectral analysis for roll waves carried out in \cite{JNRYZ} is to carry out a semiclassical limit analysis  first approximately, then exactly, diagonalizing the system into a collection of scalar modes, linked by the boundary conditions at the shocks. The idea here is that the first, approximate diagonalization can be performed in exactly the same way for the linearized evolution equations, substituting for the second step a ``Kawashima-type'' energy estimate eliminating off-diagonal terms. This last step is described in detail in Section \ref{s:sv}; see also related analyses of \cite{MZ_large,RZ,YZ}. However, again, the reader will lose little by ignoring this technical step, and simply taking the equations to be diagonal from the beginning.

This leaves us with a collection of $n$ scalar eigenmodes
\be\label{scalarmodes}
	\partial_t u_j + \alpha_j \partial_x u_j = \gamma_j u_j 
\ee
on each period $[0,X)$, coupled by boundary conditions linking $u_j(X^\pm)$, $n-1$ of which are ``transverse,'' i.e., for which $\alpha_j(x)$ has constant sign, and one distinguished mode for which $\alpha_j(x)$ has a simple zero at a unique sonic point $x_s\in [0,X]$. For all of these, coefficients depend periodically on $x$ alone. These two types present complementary difficulties, and are treated in rather different fashion. And, indeed, there is a mild incompatibility between them, leading to some unfortunate complications in exposition.

Namely, for (noncharacteristic) transverse modes, $\d_t u$ and $\d_x u$ are equivalent in $L^2$ norm up to lower-derivative terms, and so we may use these interchangeably in deriving higher-order estimates. Since $t$-derivatives pass through equations and boundary conditions, it is evident that any estimate satisfied in $L^2$ is satisfied also in $H^s$ for arbitrary $s\geq 1$; that is, we have the usual situation that spectrum is independent of the choice of norm.

For sonic modes on the other hand, as emphasized in \cite{JNRYZ,DR2}, spectrum depends very much on the norm $H^s$;
in particular, they are unstable in $L^2$ but stable in $H^1$; see Remark~\ref{sonicrmk} for further discussion on this point. Thus, we are forced to work in $H^1$ for their analysis, and this means estimating $\d_x u$, since $\d_t u$ and $\d_x u$ are no longer equivalent, $\d_t u\sim \alpha \d_x u$ degrading at the sonic point corresponding to vanishing of $\alpha$. On the good side, a closer inspection of the Lax characteristic conditions at the shocks reveals that characteristics of the sonic mode must exit the periodic cell on either end, hence this mode requires no boundary conditions. Put equivalently, boundary contributions may be expected to be favorable, as indeed they turn out to be. See Section \ref{s:sv} for further details on this point.

As modes are coupled through the boundary conditions, we must work in common coordinates, and so we analyze transverse modes, too, in the $x$-derivative coordinate $\d_x u$ rather than the more natural $\d_t u$. For these coordinates, the form of the principal part of both the interior (differential) equations and the boundary conditions changes, according to the rule 
\be\label{rule}
u \to \alpha u
\ee
imposed by $\d_t u\sim \alpha \d_x u$, somewhat obscuring the simplicity of the underlying argument, in particular the property that $H^s$ and $L^2$ stability go similarly for transverse modes. In reading through Section \ref{s:sv}, it may be helpful for the reader to keep in mind the rule of thumb \eqref{rule} in following the path of the analysis.

\subsection{Plan of the paper}\label{s:plan}
In Section \ref{s:linalg}, we present a Lyapunov-type linear algebraic lemma needed for the $n\times n$ system case, $n>2$. In Section \ref{s:sv}, we carry out in detail the analysis for the $2\times 2$ Saint-Venant equations,
for which both sonic and transverse modes are scalar, and the issues discussed in Section \ref{s:linalg} do not arise. In Section \ref{s:general}, we treat the general $n\times n$ case in a streamlined fashion, omitting technical details in common with the $2\times 2$ case in order to focus on the new issues arising in transverse modes for the system case, and give our main result in Section~\ref{s:general}. Finally, in Section \ref{s:disc}, we briefly discuss perspectives and open problems. The appendices are devoted to additional discussion from Section \ref{s:linalg}.

\medskip

{\bf Acknowledgement:} Thanks to Hari Bercovici for a helpful discussion regarding the linear algebraic lemma of Section \ref{s:linalg}, to Zhao Yang for valuable insights on a preliminary version and to Aric Wheeler for providing initial supporting numerics. L.M.R. thanks Indiana University for its hospitality during two visits funded partly by the mathematics department's short-term research visitor program.

\section{A linear algebraic lemma}\label{s:linalg}
We begin with a linear algebraic lemma that seems of interest in its own right. For a general complex $n\times n$ matrix $B$, denote by $\rho(B)$ its spectral radius, defined as the maximum modulus of its eigenvalues, and $\|B\|$ its $\ell^2$ operator norm, defined as $\max_{|u|=1}|Bu|$. Evidently, $\rho(B)\leq \|B\|$; a classical question of fundamental numerical and theoretical interest is when these two quantities coincide. A classical result is that
\[
\inf_{P\,\textrm{ invertible}}\|P\,B\,P^{-1}\|=\rho(B)\,.
\]

A variant arising in our analysis is whether a combination of scaling transformations $B\mapsto SBS^{-1}$ with $S=\diag\{s_1,\dots,s_n\}$, $s_j$ positive real, and multiplication $B\mapsto UB$ by diagonal unitary matrix $U=\diag\{u_1, \dots, u_n\}$, yields equality, $\rho(SUBS^{-1})= \|SUBS^{-1}\|$, or, more generally (noting that $\|S\|$ might run off to infinity)
\[
\inf_{S,U} \Big( \|SUBS^{-1}\|- \rho(SUBS^{-1})\Big)=0.
\]
Here and elsewhere we always assume that matrices denoted with an $S$ letter are diagonal with positive diagonal entries whereas those denoted with a $U$ are diagonal with diagonal entries of modulus $1$.

The following Lyapunov-type lemma answers this question in the affirmative for lower dimensions $n\leq 3$
but in the negative for higher dimensions $n\geq 4$.

\begin{lemma}\label{linalg_lem} 
\begin{enumerate}
\item For any $B\in\cM_n(\C)$, any $S$, $U$,
\[
\rho(UB)=\rho(SUBS^{-1}) \leq \|SUBS^{-1}\|= \|SBS^{-1}\|\,.
\]
In particular, for any $B\in\cM_n(\C)$,
\[
\inf_{S,U} \Big( \|SUBS^{-1}\|- \rho(SUBS^{-1})\Big)\,=\,
\inf_S \|SBS^{-1}\|-\max_U \rho(UB)\,.
\]
\item When $n\leq 3$ and $B\in\cM_n(\C)$ or when $n\leq 5$ and\footnote{Note that the spectrum is still the complex spectrum.} $B\in\cM_n(\R)$,
\[
\inf_{S,U} \Big( \|SUBS^{-1}\|- \rho(SUBS^{-1})\Big)\,=\,0\,.
\]
\item For any $n\geq 4$, there exists $B\in\cM_n(\C)$ such that
\[
\inf_S \|SBS^{-1}\|>\max_U \rho(UB)\,.
\]
\end{enumerate}
\end{lemma}

We shall make important use of Lemma~\ref{linalg_lem} in the treatment of hyperbolic systems in Section \ref{s:general}: specifically, in showing that sharp spectral information for a certain initial-boundary value problem (associated with some matrix $B$), encoded in a spectral radius condition $\rho(UB)<1$ for any $U$, may be realized by a Lyapunov-type energy estimate requiring $\|SBS^{-1}\|<1$ for some $S$.

In order not to delay too much the use of Lemma~\ref{linalg_lem} we postpone to Appendix~\ref{s:ceg3} the treatment of complex counterexamples when $n\geq4$ and the proof of the absence of gap in the real case $n\leq5$.

\begin{proof}
By invariance of spectrum under similarity transformations, we have 
\[
\rho(UB)=\rho(SUBS^{-1})
\]
for any $S$, $U$. Meanwhile, by invariance of norms under isometries and commutation of diagonal matrices, we have
\[
\|SUBS^{-1}\|\,=\,\|USBS^{-1}\|\,=\,\|SBS^{-1}\|
\]
for any $S$, $U$. This is sufficient to prove (1).

\medskip
\emph{Simple case.} We first prove the absence of gap when the infimum of $S\mapsto \|SBS^{-1}\|$ is reached at a $S_*$ such that the largest eigenvalue of $(S_*BS_*^{-1})^*S_*BS_*^{-1}$ is simple. Without loss of generality we may assume that $S_*$ is the identity matrix, since otherwise we may replace $B$ with $S_*BS_*^{-1}$. Recall that for any $S$, $\|SBS^{-1}\|=\sqrt{\rho((SBS^{-1})^*SBS^{-1})}$.

In this case, by standard matrix perturbation theory, near identity the largest eigenvalue of $(SBS^{-1})^*SBS^{-1}=S^{-1}B^*S^2BS^{-1}$ depends smoothly on $S$ and its $s_j$-partial derivative at identity is given by
\[
r^*\left(-E_j B^*B +  B^*(2E_j)B - B^*B E_j\right)r\,=\,
2(|(Br)_j|^2-\|B\|^2|r_j|^2)\,,
\]
where $r$ is a unit eigenvector associated with the maximal eigenvalue of $B^*B$, $E_j$ is the diagonal matrix with only nonzero entry $(E_j)_{jj}=1$ and $(Br)_j$ and $r_j$ denote $j$th entries of $Br$ and $r$. Since by assumption the identity matrix is a critical point of $S\mapsto \|SBS^{-1}\|$ we deduce that for any such $r$
\[
|(Br)_j|\,=\,\|B\|\,|r_j|\quad\textrm{ for any }j
\]
which readily implies the existence of a unitary diagonal matrix $U$ such that $UBr = \|B\|r$, giving $\rho(UB)\geq \|B\|$ and thus $\rho(UB)=\|B\|=\|UB\|$, verifying (2) in this restricted case.

\medskip
\emph{Density argument for full matrices.} We next verify (2) when $n\leq 3$ for complex matrices that are ``full'' in the sense that they have no nonzero entries. To do so, we run a continuity/density argument from the ``simple'' case proved before. To prepare the density part of the argument a few definitions are in order. We shall say that a set has dimension at most $D$ if it is the union of finitely many manifolds of dimension at most $D$. The definition is designed so that providing a parametrization by a manifold of dimension $D$ is sufficient to check that a set is at most of dimension $D$ whereas, when $D<D'$, there holds that in a manifold of dimension $D'$ the complement of a set of dimension at most $D$ is dense. 

To prove continuity in $B$ of the gap near a ``full'' matrix, we simply need to observe that near such a matrix the infimum over $S$ is actually a minimum over a fixed compact subset of $S$. This stems from the fact that if $B_{jk}\neq0$ the $j,k$-entry $(SBS^{-1})_{jk} = B_{jk}(s_j/s_k)$ goes to infinity when $s_j/s_k\to \infty$. This shows that near a ``full'' matrix the infimum is achieved with ratios $s_j/s_k$ varying in a fixed compact of $(0,\infty)$. Since $SBS^{-1}$ only depends on those ratios we may actually fix $s_1\equiv 1$ without loss of generality and thus restrict $S$ to a compact set as announced.

To conclude the analysis of the ``full'' matrix case there remains to show that is dense the set of matrices $B$ such that the infimum of $S\mapsto \|SBS^{-1}\|$ is reached at a $S_*$ such that the largest eigenvalue of $(S_*BS_*^{-1})^*S_*BS_*^{-1}$ is simple. We do so by examining the orbits under $S\mapsto S B S^{-1}$ of matrices $B$ such that $\rho(B)^2$ is a multiple eigenvalue of $B^*B$.  The space $\cM_n(\C)$ has real dimension $2n^2$. On the other hand, as we detail below, by singular-value decomposition $B=LDR^*$, $R$, $L$ unitary and $D$ diagonal, real nonnegative, the set of matrices $B$ for which $B^*B$ has a $m$-tuple largest eigenvalue in $D$ has dimension (at most) $2n^2 + 1-m^2$, hence the set of all scalings $SBS^{-1}$ (taking without loss of generality $S_{11}=1$) of such a $B$ has dimension (at most) $n-1$ higher, or (at most) $2n^2+ n-m^2$. When $n\leq 3$ there holds $2n^2+ n-m^2<2n^2$ for any $m\geq 2$. This implies that the set of ``simple'' matrices treated before is dense when $n\leq3$ and thus proves the absence of gap for ``full'' matrices when $n\leq 3$.

Before moving on to the treatment of the nonfull case let us give more details on the above dimensional count. The sets of unitary matrices $R$ and $L$ have dimensions $n^2$ apiece, accounting for $2n^2$ degrees of freedom.  On the other hand, there is an overcount\footnote{Meaning that one may reduce the dimension of the implicit parametrizing manifold.} of $(n-m)$ for each eigenvalue outside the $m$-repeated block under consideration, as multiplication of the $j$th columns of $R$ and $S$ by the same unitary complex number leaves the singular-value decomposition unchanged. Meanwhile, multiplication on the right by the same unitary $m\times m$ matrix of the $n\times m$ block of columns associated with the repeated eigenvalue also leaves the  singular-value decomposition unchanged, corresponding to an overcount of $m^2$. Subtracting these values from the count, and adding the $(n-m)+1$ parameters corresponding to the entries of $D$, and the $n-1$ parameters corresponding to different scalings $S$ (again, setting the upper left entry to $1$ without loss of generality), we obtain finally (at most) $2n^2 - (n-m)- m^2 + (n-m+1) + (n-1)= 2n^2+ n-m^2$ as claimed.

\medskip
{\em General case.} In the case that some entries of $B$ vanish, the argument is a bit more complicated and we argue by induction\footnote{Up to $n=3$.} on dimension $n$, noting that the lemma holds trivially for dimension $n=1$.

Pick $B\in\cM_n(\C)$. Let us form a directed graph with nodes indexed by integers $1\leq j\leq n$ and a connection from $i$ to $j$ if $B_{ij}\neq 0$. Note that when $i$ is connected to $j$ the infimum over $S$ may be restricted to $S$ with $s_j/s_i$ bounded from above. For any nodes $(i,j)$ in the same closed loop of this graph, all ratios $s_j/s_i$ must vary in a compact of $(0,\infty)$. Thus we identify all nodes lying in a closed loop to reduce to a single node. Note that the argument on ratios still applies to the new graph. Repeating this process, we arrive finally at an irreducible configuration, each connected component of which consists of trees.

Note that if the process ends with a single node this means that for the infimum in $S$ for $B$ and all the nearby matrices all the ratios must lie in a compact of $(0,\infty)$. From this the argument of the full case may be repeated and we conclude again to the absence of gap.

If the process ends with more than one connected component this means that up to reordering indices the matrix $B$ is block-diagonal with at least two blocks. In this case one is thus effectively reduced to lower dimension and concludes from the induction hypothesis.

There remains to deal with the case when the process ends with a single tree that is not reduced to a single node. Pick a node in the final tree. Let $\tilde{B}$ denote the matrix obtained from $B$ by zeroing out any $B_{jk}$ with either $j$ ending in the chosen node and $k$ not ending in this node or the reverse. Note that $\tilde{B}$ is effectively block diagonal and thus verifies the no-gap conclusion. Let $S_\eps$ denote the diagonal matrix with $j$th entry $1$ if $j$ lies in the chosen node, $\eps$ if $j$ lies above the chosen node in the final tree,  $\eps^{-1}$ if $j$ lies below the chosen node in the final tree. Note that $S_\eps B S_\eps^{-1}$ converges to $\tilde{B}$ when $\eps\to0$, the unbounded ratios of $S_\eps$ falling only against zero entries of $B$. This implies that
\begin{align*}
\inf_S\|S\tilde{B}S^{-1}\|&\geq\ \inf_S\|SBS^{-1}\|\,,&
\max_U \rho(U\tilde{B})&=\max_U \rho(U\,B)\,.
\end{align*}
Thus from the absence of gap for $\tilde{B}$ stems the absence of gap for $B$.

\medskip
As announced the rest of the proof is provided in Appendix~\ref{s:ceg3}.
\end{proof}

\br\label{algrmk}
The proof of Lemma \ref{linalg_lem}  affords at the same time a strategy for finding a minimal $S$
when $B$ is full, or more generally when $B$ is irreducible in the sense that the graph reduction of the final step ends into a single node. Near such a matrix, $\inf_S \|SBS^{-1}\|$ is continuous in $B$ whereas for matrices as in the \emph{simple-case} step critical points of $S\mapsto \|SBS^{-1}\|$ are global minimizers. Thus for $n\leq 3$ a gradient search on $S\mapsto\|SBS^{-1}\|$ should for generic $B$ converge toward the unique critical point and global minimizer. By contrast, the energy landscape of $U\mapsto\rho(UB)$ is rather complicated, possessing in general multiple critical points, including saddlepoints and local maxima along with a global maximum; see Appendix~\ref{s:UB}.
\er

\section{The Saint-Venant equations}\label{s:sv}
We illustrate first our damping construction for the simplest case of the $2\times 2$ Saint-Venant model 
treated in \cite{JNRYZ}. For such a low-dimensional system the linear algebraic part is trivial because it effectively involves a $1\times 1$ matrix. 

Explicitly, we consider the $2\times 2$ Saint-Venant model
\be\label{SV}
\d_t(f_0(W)) + \d_x(f(W))=R(W), \quad W=(h,U)^T,
\ee
where 
\begin{align*}
f_0(W)&=(h,hU)^T\,,&
f(W)&=\left(hU,hU^2+\frac{h^2}{2F^2}\right)^T\,,&
R(W)&=(0,h-|U|\,U)^T\,,
\end{align*}
with $F$ a Froude number. We write the system in abstract form so as to prepare the transition to the general case dealt with in the next section. Yet we borrow from \cite{JNRYZ} --- that relies on the particular form of the system --- the verification of the needed structural assumptions made explicit in the next section.

The only restrictions we make as assumptions are that the notion of solution we use is entropic solution and that we consider only periodic waves with a single shock by period.

\subsection{Spectral analysis}\label{s:wkb}

Let us first recall the derivation of high-frequency spectral asymptotics carried out in \cite{JNRYZ}.

Starting with \eqref{SV}, linearizing in co-moving coordinates $\tilde x=x-ct$ about a periodic traveling wave profile $\bar W$ moving at speed $c$ with period $X$, and changing unknown to a ``good unknown'' (see Remark~\ref{goodrmk} below), we obtain eigenvalue equations 
\be \label{evalsv}
(Aw)'= (-\lambda A_0 +E)w, \qquad 
\ee
where 
\begin{align*}
A&=\dD f(\bar W)-c\,A_0,&A_0&=\dD f_0(\bar W),&E&=\dD R(\bar W),
\end{align*}
augmented with jump conditions 
\be\label{goodjump}
y_j\,(\lambda \,[f_0(\bar W)] - [R(\bar W)])=-[Aw]_{jX},
\ee
at the shocks, where $y_j$ is (the Laplace transform of) an unknown shift in shock location to be determined in the course of solving \eqref{goodjump}. We have translated to force shocks to be at $jX$, $j\in\Z$, and denoted $[u]_d=u(d^+)-u(d^-)$, dropping $d$ subscript when $d\in X\Z$ and $u$ has period $X$.

Here, $\dD f(\bar W)$ and $\dD f_0(\bar W)$ are exactly as for isentropic polytropic gas dynamics with $\gamma=2$.
In particular, the eigenvalues of $A_0^{-1}A$ consist of acoustic modes $\alpha_1$, $\alpha_2$ that are distinct for all choices of $W$, satisfying  $\alpha_1<U-c<\alpha_2$. See \cite[Section~4]{JNRYZ} for further details.

Performing a ``frozen-coefficients'' diagonalization procedure as in \cite[Section~6]{JNRYZ}, 
we obtain
\be\label{diagsvY}
\bp \alpha_1&0\\0&\alpha_2\ep u'=-\bp \lambda +\gamma_1 & \beta_1 \\ \beta_2 & \lambda+ \gamma_2   \ep u, 
\ee
with $\beta_j$, $\gamma_j$ real. Namely, here, $w=Tu$ where $T$ is such that $T^{-1}A_0^{-1}AT$ is diagonal with diagonal entries $\alpha_1$, $\alpha_2$. We refer to \cite[Section~6]{JNRYZ} for explicit expressions\footnote{Our present notation differ though.}. As shown in \cite[Section~2]{JNRYZ} discontinuities of roll-waves of \eqref{SV} are necessarily Lax $2$-shocks, 
\be\label{Laxalpha}
\alpha_1(X^-)<0<\alpha_2(X^-)\qquad\textrm{ and }\qquad \alpha_1(0^+)<\alpha_2(0^+)<0\,.
\ee

Next, we perform a further diagonalization $u=\tilde{T}z$, with $\tilde{T}=\Id + \cO(\lambda^{-1})$ eliminating to $\cO(\lambda^{-1})$ the off-diagonal entries $\beta_j$, obtaining
\be\label{diagsv}
\bp \alpha_1&0\\0&\alpha_2\ep z'=-\bp \lambda+\gamma_1 & 0 \\ 0& \lambda+ \gamma_2   \ep z +\cO(\lambda^{-1})z\,.
\ee
As observed in \cite[Appendix~A]{JNRYZ}, the ``sonic mode'' associated with the characteristic $\alpha_2$ that changes sign at a $x=x_s$, blows up as $x\to x_s$ from both sides. This mode is therefore not in $H^1_{loc}$, and is not included in the periodic Evans-Lopatinsky determinant for the problem, which tests whether solutions of the interior eigenvalue ODE above satisfy the boundary conditions corresponding to the linearized Rankine-Hugoniot condition at the inviscid shock in $w$ coordinates, up to a Floquet shift $\eD^{\iD\xi X}$, $\xi\in \R$: namely, the function
\be\label{EL}
D(\lambda,\xi):=\det \bp (Aw_\lambda)(X^-)- \eD^{\iD\xi X} (Aw_\lambda)(0^+) & \lambda[f_0(\bar W)] - [R(\bar W)]\ep,
\ee
where $w_\lambda$ is the unique (up to normalization) analytic solution of the interior eigenvalue ODE. 

The Evans-Lopatinsky determinant is related to a suitably defined spectrum with $(w,(y_j))$ measured in $H^1(\tRR)\times \ell^2(\Z)$ where
\[
\tRR:=\bigcup_{j\in\Z}(jX,(j+1)X)\,,
\]
through the use of a Floquet-Bloch transform that reduces the whole line problem to periodic problems parametrized by a Floquet exponent $\xi$. For a detailed description/derivation, see again \cite[Section~4]{JNRYZ}.

For large $|\lambda|$, then, the Evans-Lopatinsky determinant is given asymptotically by\footnote{We use here $\sim$ in a rather informal way.}
\be\label{svas}
D(\lambda,\xi)\sim \det 
\bp \eD^{-\int_0^X\frac{\lambda+\gamma_1}{\alpha_1}}(AT_1)(X^-) -\eD^{\iD\xi X}(AT_1)(0^+) & \lambda\,[f_0(\bar W)] \ep
\,,
\ee
where $T_1(x)$ is the first column of $T(x)$, i.e., the specific eigenvector of $(A_0^{-1}A)(x)$ associated with eigenvalue $\alpha_1$ that is used in the coordinatization of $w$ by $u_1$, $u_2$. In particular when $\Re(\lambda)$ itself is large
\[
D(\lambda,\xi)\sim \lambda\,\eD^{-\int_0^X\frac{\lambda}{\alpha_1}}\det 
\bp (AT_1)(X^-)& [f_0(\bar W)] \ep\,,
\]
and one recovers that local well-posedness near Lax shock requires the Lopatinsky condition 
\[
\det\bp (AT_1)(X)& [f_0(\bar W)] \ep\,\neq\,0\,.
\]

What we seek to identify is the condition ensuring that for $\Re \lambda\geq -\eta$, $\eta>0$ fixed sufficiently small and $|\lambda|$ large, $D(\lambda,\xi)$ does not vanish for any $\xi\in \R$. Directly evaluating the principal part of \eqref{svas}, we obtain 
\be\label{indeq}
D(\lambda,\xi)\sim
 \lambda\,\eD^{-\int_0^X\frac{\lambda+\gamma_1}{\alpha_1}}\det 
\bp (AT_1)(X^-)& [f_0(\bar W)] \ep
 \,\Big(1-\eD^{\iD\xi X} \eD^{\int_0^X\frac{\lambda+\gamma_1}{\alpha_1}}C\Big)\,,
\ee
where
\be\label{C}
C= \frac{\det \bp (AT_1)(0^+) & [f_0(\bar W)] \ep}{\det \bp (AT_1)(X^-) & [f_0(\bar W)]\ep}.
\ee
The occurrence of a high-frequency spectral gap is thus equivalent to
\be\label{ind}
I:=
e^{\int_0^X\frac{\gamma_1}{\alpha_1}}\,C\,<\,1.
\ee
The quantity $I$ is referred to in \cite{JNRYZ} as the ``high-frequency stability index''. For later reference, we give an alternative characterization of $C$. It follows from the Lopatinsky condition that there exists a unique $(a,b)$ such that
\[
a\,(AT_1)(X^-)+b\,[f_0(\bar W)]\,=\,(AT_1)(0^+)
\]
and from the Cramer formula that $a=C$.

In \cite{JNRYZ} the high-frequency spectral gap condition \eqref{ind} was checked numerically, but not proved, to hold for all roll-waves of \eqref{SV}.

\subsection{Linear damping estimate}\label{s:linsv}

Even for smooth periodic waves the best notion of stability one may expect, coined as space-modulated stability in \cite{JNRZ13}, is encoded by the control of $(v,\d_x\psi,\d_t\psi)$ such that
\be\label{og}
v(x,t)= W(x+ct-\psi(x,t),t) -\bar W(x)\,,
\ee
with $v$, $\psi$ designed to capture respectively shape and position deformations. We refer to \cite{R_HDR,R_Roscoff} for further discussion on the latter. For discontinuous waves, when measuring proximity in piecewise smooth topologies one must already introduce a similar treatment in local-well-posedness results; see for instance \cite{Majda_stability_multiD_shock-fronts,N3}. 

Inserting \eqref{og} in \eqref{SV} in co-moving coordinates and considering $(v,\psi)$ as unknown small functions 
gives the linearized equations 
\be \label{linsv0}
A_0 \d_tv + \d_x(Av)- Ev=-\d_t\psi A_0\,\bar W'-\d_x\psi A\,\bar W'+F_{quad},\qquad\textrm{on }\tRR 
\ee
with $A_0$, $A$, $E$ as in \eqref{evalsv}, together with linearized jump conditions 
\be\label{abjump0}
\d_t\psi\,[f_0(\bar W)]+[Av]= G_{quad}\,,\qquad\textrm{on }X\Z\,,
\ee
with $F_{quad}$ and $G_{quad}$ replacing terms that are at least quadratic.

Besides the dropping of forcing terms and the Laplace transform in time, there is another difference between variables $(v,\psi)$ of the present subsection and $(w,y)$ of the preceding subsection. Let us set
\begin{align*}
w(x,t)&=v(x,t)+\psi(x,t) \bar W'(x),&
y(jX,t)&=\psi(jX,t),
\end{align*} 
and note that \eqref{linsv0}-\eqref{abjump0} becomes 
\begin{align*}
A_0 \d_tw + \d_x(Aw)- Ew&=F_{quad},\qquad\textrm{on }\tRR\,,\\
\frac{\dD\ y}{\dD t}\,[f_0(\bar W)]
-y\,[R(\bar W)]+[Aw]&=G_{quad}\,,\qquad\textrm{on }X\Z\,.
\end{align*}

\br\label{goodrmk}
The ``good unknown'' $w$ is used crucially in proofs of nonlinear and linear asymptotic stability of spectrally stable periodic waves to isolate the phase $\psi$ in a manner yielding optimal linear bounds on the residual $v$. See for instance \cite{JNRZ-RD1,JNRZ13,R_linearKdV}. In related nonlinear schemes, time decay and improvement in nonlinearity is obtained in $(v,\psi)$ whereas the $w$ variable is used to close in regularity without losing or gaining in nonlinearity or decay, through high-frequency damping estimates. The point is that some of the terms in $F_{quad}$ and $G_{quad}$ contain too many derivatives to be considered as forcing terms in the energy estimate but may instead be thought as introducing small variations in coefficients of the linear part and treated perturbatively as such. For the sake of simplicity we shall simply obtain bounds omitting the fact that $(F_{quad},G_{quad})$ are replacing terms with too many derivatives. We refer to \cite{YZ,FR1,FR2} for complete treatments for discontinuous fronts with a single discontinuity.
\er

In the following we estimate the regularity of $w$ when $(w,y)$ solves
\be \label{linsv}
A_0 \d_tw + \d_x(Aw)- Ew=F,\qquad\textrm{on }\tRR 
\ee
together with  
\be\label{abjump}
\frac{\dD\ y}{\dD t}\,[f_0(\bar W)]
-y\,[R(\bar W)]+[Aw]=G\,,\qquad\textrm{on }X\Z\,,
\ee
for some given $(F,G)$.

\smallskip

Our goal is to establish the following linear damping estimate, converting high-frequency spectral
information into a Lyapunov-type energy estimate.

\bpr\label{ldampingsv}
Assume high-frequency spectral stability \eqref{ind}. For some $\theta>0$ and $C$, if $w$ solves \eqref{linsv}-\eqref{abjump} (for some shift sequence $y$) on a time interval $[0,T_0]$ with $w(0,\cdot)=w_0$ then for all $0\leq t\leq T_0$
\begin{align}
\label{ldampest}
\|w(\cdot,t)\|_{{H^1(\tRR)}}^2
&\leq C\,\eD^{-\theta t} \|w_0\|_{H^1(\tRR)}^2\\\nonumber
&\quad+C\int_0^t \eD^{-\theta(t-\tau)}\big(\|F(\cdot,\tau)\|_{H^1(\tRR)}^2
 +\|(G,\d_tG)(\cdot,t)\|_{\ell^2(X\Z)}^2\big)\, \dD\tau
\\\nonumber
&\quad+C\int_0^t \eD^{-\theta(t-\tau)}\big(\|w(\cdot,\tau)\|_{L^2(\tRR)}^2
 + \|y(\cdot,\tau)\|_{\ell^2(X\Z)}^2
 \big)\, \dD\tau.
\end{align}
\epr

\br\label{regrmk}
In related nonlinear analyses it is actually important to note that the linear slaving of the $H^1$-norm to the $L^2$-norm may be improved into linear slaving of $H^s$ to $L^2$ for any $s$ sufficiently large since its nonlinear counterpart must be carried out with $s$ such that $H^s$ is embedded in $W^{1,\infty}$. Again we refer to \cite{YZ,FR1,FR2} for the necessary adaptations. A less usual phenomenon appears here: because of the presence of a characteristic point there is a lower bound on the indices $s$ for which a linear slaving of $H^s$ to $L^2$ is possible. The same regularity threshold appears for the spectral problem. For \eqref{SV} this $H^s$-threshold is $s>1/2$.
\er

\begin{proof}
In spirit we follow the ``gauge'' approach of \cite{RZ}; however the concrete way of finding an appropriate gauge is rather different here. In \cite{RZ}, we looked for smooth periodic weights and obtained those by solving scalar linear differential equations with periodic coefficients of zero integral so as to enforce spatial periodicity. Here, periodicity is imposed by force, since the problem is discontinuous, consisting of cells connected by boundary conditions only. Integral constraints on the gauge arise in this case rather through interaction between boundary conditions on either end of the cell. In terms of technical specifics, we use in noncharacteristic modes the ``Goodman'' weights of \cite{MZ_large,YZ,FR2}, and in the characteristic, or ``sonic'' mode, we use a key technique originating from \cite{DR2,GR} to treat the sonic mode by $H^1$ energy estimate. These two types of diagonal weights are combined with a ``Kawashima'' type estimate as in \cite{MZ_large,YZ,FR1,FR2} removing cross-terms corresponding to matrix entries $\beta_j$; this gives a different way of removing ``lower-order'' terms $\beta_j$, translating into energy estimates the symbolic calculations of the high-frequency spectral analysis above.

Making the same ``frozen-coefficients'' coordinatization as in \eqref{diagsvY} converts \eqref{linsv} to
\be\label{lindiag}
\d_tu + \bp \alpha_1 & 0 \\ 0 & \alpha_2 \ep \d_x u =-\bp \gamma_1 & \beta_1 \\ \beta_2 & 
\gamma_2\ep u
+T^{-1}A_0^{-1}\,F\,,
\ee
$u=(u_1,u_2)^T$, with the single boundary condition
\be\label{bclin}
u_1(jX^-,t)=a_0 u_1(jX^+,t)+ b_0 u_2(jX^-,t)+c_0 u_2(jX^+,t)+d_0 G(jX,t)+e_0 y(jX,t)
\ee
obtained from \eqref{abjump} by invoking Lopatinsky condition. We note, crucially, that $a_0$ is identical to $C$ defined in \eqref{C}. By using $\d_tu_j=\alpha_j \d_xu_j+\lot$ (where here and henceforth $\lot$ denotes lower order terms), together with the time-derivative of \eqref{bclin}, we obtain the differentiated boundary conditions
\be\label{dbclin}
\alpha_1\d_x u_1(jX^-,t)= a_0\alpha_1\d_x u_1(jX^+,t)
+b_0\alpha_2\d_x u_2(jX^-,t)+ c_0 \alpha_2\d_x u_2(jX^+,t)+\lot\,.
\ee
This is to be combined with 
\[
\d_t(\d_xu)+ \bp \alpha_1 & 0 \\ 0 & \alpha_2 \ep \d_x(\d_x u)=-\bp \gamma_1+\alpha_1' & \beta_1 \\ \beta_2& 
\gamma_2+\alpha_2'\ep \d_xu+\lot\,.
\]
Before proceeding we warn the reader that though both in interior and boundary equations $\lot$ simply covers terms with fewer spatial derivatives of the unknown $u$ or data terms, performing the underlying reduction in the boundary equations requires a few more manipulations including the replacement of derivatives of $y$ by resorting to \eqref{abjump}. 

For $k=1,2$, we introduce periodic scalar weights $\Omega_k=\eD^{\omega_k}$ on $\tRR$ with $\omega_k$ to be determined. Then, omitting to mark the $t$-dependency, we compute
\begin{align*}
\frac12\frac{\dD}{\dD t}&\langle \Omega_k \d_xu_k,\d_xu_k \rangle_{L^2}
+\langle \Omega_k \delta_k\d_xu_k,\d_xu_k \rangle_{L^2}\\
&+\frac12\sum_{j\in\Z}\left((\alpha_k \Omega_k (\d_x u_k)^2) (jX^-) 
-(\alpha_k \Omega_k (\d_x u_k)^2) ((j-1)X^+)
\right)\\
&=-\langle \Omega_k \d_x u_k, \beta_k \d_x  u_{k'} \rangle_{L^2}+\lot
\end{align*}
with $k'$ denoting the index complementary to $k$ and
\be\label{deltaj}
\delta_k:= \frac12 \alpha_k'+\gamma_k-\frac12\alpha_k\omega_k'\,.
\ee

In order to have damping, we need in the final computation that both interior and boundary terms exhibit good signs. Negativity of diagonal interior terms corresponds to positivity of dissipation coefficients. We arrange this in different ways for the sonic mode $k=2$ and the transverse mode $k=1$. 

\smallskip

{\em Sonic mode.} For the sonic mode, we observe that at the sonic point $x_s$ for which $\alpha_2(x_s)=0$, we have $\alpha_2'(x_s)>0$, so that $\delta_2(x_s)>0$, a consequence of the Lax conditions $\alpha_2(0)<0<\alpha_2(X)$ and the fact that there is only a single sonic point/change in sign of $\alpha_2$ on $[0,X]$. Again we refer to \cite[Section~2]{JNRYZ} for a proof that those hold for roll-waves of \eqref{SV}. Indeed, positivity of $\alpha_2'(x_s)$ is a necessary condition for our scheme to work. We also observe that 
\[
\frac12\alpha_2'(x_s)+\gamma_2(x_s)>0\,.
\]
This positivity is the reason why we could set the spectral problem at the $H^1$ level compatible with stability; the presence of a characteristic point making the nature of the spectrum extremely sensitive to regularity. The positivity of $\alpha_2'(x_s)$ is however sufficient to guarantee that this occurs in $H^s$ for $s$ sufficiently large, the stopping criterion being
\[
\left(s-\frac12\right)\,\alpha_2'(x_s)+\gamma_2(x_s)>0\,.
\]
Again we refer to \cite[Appendix~A]{JNRYZ} and \cite{DR2} for details.

Accordingly, drawing inspiration from \cite{DR2,GR}, we choose $\omega_2$ so as to enforce
\[
\delta_2\equiv \frac12\alpha_2'(x_s)+\gamma_2(x_s)>0,
\]
by the well-defined quadrature formula
\[
\omega_2' = \frac{\alpha_2'-\alpha_2'(x_s)}{\alpha_2}+2\frac{\gamma_2-\gamma_2(x_s)}{\alpha_2}
\]
or, for $x\in ((j-1)X,jX)$, 
\be\label{2ode}
\Omega_{2,(C_0)}(x)\,:=\,C_0 \exp\left(\int_{(j-1)X}^x \left( 
 \frac{\alpha_2'(y)-\alpha_2'(x_s)}{\alpha_2(y)} + 2\frac{\gamma_2(y)-\gamma_2(x_s)}{\alpha_2(y)}\right)\dD y\right)\,.
\ee
Boundary terms for the sonic mode $k=2$ are individually of good sign. Choosing the constant $C_0>0$ in \eqref{2ode} sufficiently large, we may make these favorable boundary terms as large as we please, hence any other boundary terms involving the $k=2$ mode may be absorbed, hence can be neglected in our further computations and shall be considered as $\lot$ below.

\smallskip

{\em Transverse mode.} For the transverse mode $k=1$, we have, on the other hand $\alpha_1<0$ for all $x\in [0,X]$. Setting
\[
\delta_{1,(\eps)}\equiv \eps>0,
\]
therefore, for arbitrary $\eps>0$, we obtain a well-defined quadrature
\be\label{1ode}
\omega_{1,(\eps)}'=\frac{\alpha_1'+2(\gamma_1-\eps)}{\alpha_1}, 
\ee
in which we can take $\eps>0$ as close to zero as needed in the analysis. Here, it is critical to have the precise form of $\omega_{1,(\eps)}$, as it links ``good'' boundary terms at one (outgoing) end to ``bad'' (incoming) terms at the other. For later use, we record this as, for $x\in ((j-1)X,jX)$, 
 \be\label{limeps}
 \Omega_{1,(\eps)}(x)\,:=\,|\alpha_1(x)|\,\eD^{\int_{(j-1)X}^x 2\frac{\gamma_1(y)-\eps}{\alpha_1(y)}\,\dD y}\,.
 \ee
 
With this choice of $\Omega_1$, recalling from \eqref{dbclin} that $\alpha_1\d_x u_1(jX^-)= a_0\alpha_1\d_x u_1(jX^+)+\lot$, one deduces that the boundary terms for the transverse $k=1$ are
\[
\frac12\sum_{j\in\Z}\eta_{1,(\eps)}\,(\Omega_{1,(\eps)}|\alpha_1|(\d_x u_1)^2)(jX^+) 
+\lot
\]
with 
\[
\eta_{1,(\eps)}:=1-a_0^2\frac{(\alpha_1^{-1}\Omega_{1,(\eps)})(X^-)}{(\alpha_1^{-1}\Omega_{1,(\eps)})(0^+)}
\,=\,1-a_0^2
\eD^{\int_0^X 2\frac{\gamma_1(y)-\eps}{\alpha_1(y)}\,\dD y}
\,,
\]
which is dissipative so long as $\eta_{1,(\eps)}>0$. At the limit $\eps\to 0$, $\eta_{1,(0)}>0$ is equivalent to the high-frequency  spectral stability condition \eqref{ind} found in Subsection~\ref{s:wkb}. Positivity then persists for sufficiently small $\eps>0$. {\it This is the key observation linking energy estimates to WKB-type expansion.} 

\smallskip

{\em Cross-terms.} At this point we have accounted for all except the cross-term contribution. The algebraic computation is essentially identical to the $\tilde{T}$-elimination step of Subsection~\ref{s:wkb}. As there the point is that commutators of a general matrix with $\diag(\alpha_1,\alpha_2)$ may generate any off-diagonal matrix. This is used to pick a skew-symmetric compensator $K$ to insert in 
\begin{align*}
\frac{\dD}{\dD t}\langle \d_xu,K u\rangle_{L^2}
&=\langle \d_xu,K \d_tu\rangle_{L^2}-\langle \d_tu,K \d_xu\rangle_{L^2}+\lot\\
&=\langle \d_xu,[K,\diag(\alpha_1,\alpha_2)]\,\d_xu\rangle_{L^2}+\lot
\end{align*}
so that any cross-term may be discarded.

\smallskip

The proof is then concluded by estimating the time evolution of
\be\label{E}
\cE(u):=\frac12\langle D\d_xu, \d_xu\rangle+\langle \d_xu,K u\rangle_{L^2}+\frac12 C_0'\|u\|_{L^2}^2
\ee
with $D=\diag(\Omega_1,\Omega_2)$ and choices made in the following order: 
\begin{enumerate}
\item one first picks $\Omega_1=\Omega_{1,(\eps)}$ with $\eps>0$ sufficiently small to enforce $\eta_{1,(\eps)}>0$ by benefiting from \eqref{ind};
\item then one sets $\Omega_2=\Omega_{2,(C_0)}$ with $C_0$ sufficiently large to guarantee positivity except for cross terms;
\item then one designs $K$ to kill remaining cross-terms;
\item finally one takes $C_0'$ sufficiently large to enforce that $\cE(u)$ is equivalent to $\|v\|_{H^1(\tRR)}^2$.
\end{enumerate}
\end{proof}

\subsection{Towards nonlinear stability}\label{s:nonlin}

Unlike \cite{RZ} that came to complete and improve the high-frequency part of a full nonlinear stability scheme of proof, we warn the reader that this is not the case for the present contribution. Even for the Saint-Venant system, the question is largely open.

A key difficulty already pointed out in \cite{JNRYZ} is that near periodic traveling waves there is an infinite-dimensional family of traveling waves. Despite the resolution of a similar but distinct degeneracy in \cite{GR}, the resolution of the latter by a suitable phase isolation remains unclear.


We leave this issue for future investigation.

\section{General systems}\label{s:general}
Recent refinements of the Saint-Venant equations proposed by Richards-Gavrilyuks \cite{RG1} are of larger system size, involving additional variables modeling bottom- and shock-layer vorticities. As described further in \cite{RG1,RYZ1,RYZ2,RYZ3}, these include both the full $4\times 4$ system (RG4) and and an invariant $3\times 3$ subsystem (RG3) with bottom vorticity held constant, containing in particular the class of roll wave solutions. This motivates the extension of the analysis of Section \ref{s:sv} to discontinuous periodic traveling waves of general $n\times n$ systems
\be\label{gensys}
\d_t(f_0(W)) + \d_x(f(W))=R(W).
\ee
Again, we restrict to periodic waves (of profile $\bar{W}$, period $X$ and speed $c$) with a single discontinuity by periodic cell.

We make the assumptions, satisfied for both (SV) and (RG3), that
\begin{enumerate}[label=(\roman*)]
\item System~\eqref{gensys} is {\it strictly hyperbolic} in the sense that eigenvalues of $(\dD f_0^{-1}\dD f)(W)$
for $W$ near the range of $\bar W$ are real and simple;
\item The profile $\bar W$ of period $X$ is smooth on $\tRR=\cup_{j\in\Z}((j-1)X,jX)$, with {\it admissible Lax-type} shock discontinuities at points $jX$, in the sense that $(n+1)$ characteristics enter and $(n-1)$ leave the shocks from either side and Majda's Lopatinsky condition;
\item On $(0,X)$ a single eigenvalue of $((\dD f_0)^{-1}\dD f)(\bar W)-c\,\Id$ changes sign and it does so at a single point $x_s$ in a simple way.
\end{enumerate}
Some form of $(ii)$ is required to guarantee local well-posedness whereas some form of $(iii)$ is required to discard high-frequency instabilities due to characteristic points. The former, in the present 1-D case, is precisely the condition that the Rankine-Hugoniot conditions, after eliminating the shock-location unknown, yield a rank $(n-1)$ boundary condition expressing outgoing modes in terms of incoming modes. Assumption $(i)$ is made for simplicity and it is probably the main one to relax from an applicative point of view.

\br\label{sonicrmk}
A wealth of intuition on characteristic points may be gained from the toy local spectral problem
\[
d_{(s)}\,(\cdot-x_s)\,u'=-(\lambda+\gamma_{(s)})u+F\qquad \textrm{on a neighborhood of }x_s
\]
whose $k$th derivative yields 
\[
d_{(s)}\,(\cdot-x_s)\,(u^{(k)})'=-(\lambda+\gamma_{(s)}+k\,d_{(s)})u^{(k)}+F^{(k)}\,.
\]

As in \cite[Section~2.2]{DR2} one may check that for any $\ell$, $\delta_{x_s}^{(\ell)}$ solves the formally dual eigenvalue equation when $\lambda=-(\gamma_{(s)}+\ell\,d_{(s)})$ so that $-(\gamma_{(s)}+\ell\,d_{(s)})$ belongs to the $H^k$-spectrum whenever $k>\ell+\frac12$. This implies that when $d_{(s)}<0$, for any $\eta\in\R$ there exists $k_0$ such that for any $k\geq k_0$, the $H^k$-spectrum meets the half-plane $\Re(\lambda)\geq\eta$. At the nonlinear level in the case $d_{(s)}<0$, one may prove as in \cite[Section~2.2]{DR2} the formation of new shocks in finite-time even for small smooth perturbations.

Focusing now on the case $d_{(s)}>0$, when modes are outgoing and no boundary 
condition is required, one may check as in \cite[Appendix~A]{JNRYZ} by direct computation that unique solvability in $H^k$ is equivalent to 
\[
\frac{\Re(\lambda)+\gamma_{(s)}}{d_{(s)}}>\frac12-k
\]
and $\lambda\neq-(\gamma_{(s)}+\ell\,d_{(s)})$ for any integer $\ell$ such that $\ell<k-1/2$.

Those insights are translated as in \cite{JNRYZ,DR2} to the genuine coupled and varying-coefficient case with essentially no change in the high-frequency regime.
\er

\subsection{Spectral analysis}

The eigenvalue system has the same structure as for the Saint-Venant system
\[
(Aw)'= (-\lambda A_0 +E)w, \qquad A=\dD f(\bar W)-c\,A_0,\quad A_0=\dD f_0(\bar W),\quad E=\dD R(\bar W)
\]
augmented with jump conditions 
\[
y_j\,(\lambda \,[f_0(\bar W)] - [R(\bar W)])=-[Aw]_{jX},
\]
with $(y_j)_j$ to be determined jointly with $w$.

We may perform a first diagonalization $w=T^{-1}u$, to derive 
\[
\alpha\,u'\,=\,-(\lambda\Id+\gamma+\beta)\,u
\]
with $\alpha$ and $\gamma$ diagonal and $\beta$ off-diagonal. For some $0\leq m\leq n-1$, one may order coordinates to enforce 
\begin{align*}
\alpha&=\diag(\alpha_1,\cdots,\alpha_n)\,,&
\end{align*}
with
\begin{align*}
\alpha_1&>\cdots>\alpha_m>0\,,&
\alpha_{m+1}&<\cdots<\alpha_{n-1}<0\,,&\\
\alpha_{n-1}&<\alpha_n<\alpha_m\,,&
\alpha_n(0)&<0<\alpha_n(X)\,.
\end{align*}
Accordingly we set 
\begin{align*}
\alpha_+&=\diag(\alpha_1,\cdots,\alpha_m)\,,&
\alpha_-&=\diag(\alpha_{m+1},\cdots,\alpha_{n-1})\,,&\alpha_{(s)}&=\alpha_n\,,\\
\gamma_+&=\diag(\gamma_1,\cdots,\gamma_m)\,,&
\gamma_-&=\diag(\gamma_{m+1},\cdots,\gamma_{n-1})\,,&\gamma_{(s)}&=\gamma_n\,,\\
u_+&=(u_1,\cdots,u_m)^T\,,&
u_-&=(u_{m+1},\cdots,u_{n-1})^T\,,&u_{(s)}&=\gamma_n\,.
\end{align*}

Picking $k$ such that 
\be\label{kreg}
k>\frac12-\frac{\gamma_{(s)}(x_s)}{\alpha_{(s)}'(x_s)}\,,
\ee
one then derives as in \cite{JNRYZ} a consistent $H^k$-spectral theory with spectrum in an open half plane containing $[0,+\infty)+\iD\R$ coinciding with $\lambda$ such that
\begin{align*}
&D(\lambda,\xi):=\\
&\det \bp (Aw_{1,\lambda})(X)-\eD^{\iD\xi X} (Aw_{1,\lambda})(0)
&\hspace{-0.5em}\cdots\hspace{-0.5em}& 
(Aw_{n-1,\lambda})(X)-\eD^{\iD\xi X} (Aw_{n-1,\lambda})(0)
& \lambda[f_0(\bar W)] - [R(\bar W)]\ep
\end{align*}
vanishes for some Floquet exponent $\xi\in\R$, where $(w_{1,\lambda},\cdots,w_{n-1,\lambda})$ forms a basis of the set of analytic solutions to the interior eigenvalue ODE. 

By performing a further approximate diagonalization step $v=(\Id+\cO(\lambda^{-1}))u$, one may replace $\beta$ with an $\cO(\lambda^{-1})$-term and check that a suitable choice of $(w_{1,\lambda},\cdots,w_{n-1,\lambda})$ enforces that for $\lambda$ large, $D(\lambda,\xi)$ is given at leading order by
\[
\det 
\bp \eD^{-\int_0^X\frac{\lambda+\gamma_1}{\alpha_1}}(AT_1)(X) -\eD^{\iD\xi X}(AT_1)(0)& 
\hspace{-0.5em}\cdots\hspace{-0.5em}&
\eD^{-\int_0^X\frac{\lambda+\gamma_{n-1}}{\alpha_{n-1}}}(AT_{n-1})(X) -\eD^{\iD\xi X}(AT_{n-1})(0) &
\lambda\,[f_0(\bar W)] \ep
\]
where $T_j(x)$ is the $j$th column of $T(x)$. In particular when $\Re(\lambda)$ is large, $D(\lambda,\xi)$ is given at leading order by a nonvanishing factor times
\[
\Delta_0:=
\det 
\bp (AT_1)(0)&\hspace{-0.5em}\cdots\hspace{-0.5em}&(AT_m)(0)&(AT_{m+1})(X)&\hspace{-0.5em}\cdots\hspace{-0.5em}&(AT_{n-1})(X)&[f_0(\bar W)] \ep\,.
\]

The non vanishing of the foregoing determinant is precisely the Lopatinsky condition. Therefore there exist a $C\in\cM_{n-1}(\R)$ and $\ell\in \cM_{1,n-1}(\R)$ such that
\begin{align*}
\bp (AT_1)(0)&\hspace{-0.5em}\cdots\hspace{-0.5em}&(AT_m)(0)&(AT_{m+1})(X)&\hspace{-0.5em}\cdots\hspace{-0.5em}&(AT_{n-1})(X)&[f_0(\bar W)]\ep
\bp C\\\ell\ep\\
=\bp (AT_1)(X)&\hspace{-0.5em}\cdots\hspace{-0.5em}&(AT_m)(X)&(AT_{m+1})(0)&\hspace{-0.5em}\cdots\hspace{-0.5em}&(AT_{n-1})(0)\ep\,.
\end{align*}
Inserting this in the above yields after elementary matrix reduction to the fact that for $\lambda$ large, $D(\lambda,\xi)$ is given at leading order by
\begin{align*}
\lambda\,\Delta_0\,\det(-\eD^{-\int_0^X(\lambda+\gamma_-)\alpha_-^{-1}})\,\eD^{m\iD\xi X}
\quad\det(B_{\lambda,\xi}-\Id)
\end{align*}
where 
\begin{align*}
B_{\lambda,\xi}\,:=\,\diag(\eD^{-\iD\xi X}\eD^{-\lambda\int_0^X\alpha_+^{-1}},
\eD^{\iD\xi X}\eD^{\lambda\int_0^X\alpha_-^{-1}})\,B
\end{align*}
and
\begin{align*}
B:=\diag(\eD^{-\int_0^X\gamma_+\alpha_+^{-1}},
\eD^{\,\int_0^X\gamma_-\alpha_-^{-1}})\,C\,.
\end{align*}

Since $B_{\lambda,\xi}$ only provides a leading-order description, we need some margin in the description of a high-frequency gap in terms of $B_{\lambda,\xi}$. Whereas $B_{\lambda,\xi}$ converges to zero in the large $\Re(\lambda)$ limit providing some form of compactness in values, the large $\Im(\lambda)$ limit is more complicated. The following lemma elucidates this point in ``generic'' situations.

\begin{lemma}\label{Ulem}
Under the ``generic'' assumption that
\be\label{ratcond}
\left(1,\int_0^X \alpha_1 ^{-1},\cdots,\int_0^X \alpha_{n-1} ^{-1}\right)\textrm{ are rationally independent}
\ee
then the following conditions are equivalent	
\begin{enumerate}
\item There exist $\eta>0$, $M>0$ and $c>0$ such that for any $\lambda$ such that $|\lambda|\geq M$, $\Re(\lambda)\geq -\eta$ and for any $\xi\in\R$, $|\det(B_{\lambda,\xi}-\Id)|\geq c$.\\
\item For any $a\in[0,+\infty)$,
\[
\max_{U\textrm{ diagonal, unitary}}\rho(U\,B_{a,0})<1\,.
\]
\end{enumerate}
\end{lemma}

\begin{proof}
Kronecker's theorem implies that from \eqref{ratcond} stems that for any $a$ and $M$, $B_{a+\iD \zeta,\xi}$ for $(\xi,\zeta)\in\R^2$, $|\zeta|\geq M$ is a dense subset of the set of all $U\,B_{a,0}$. This implies the lemma when combined with the observation derived by sending $a\to+\infty$ that $1\notin \sigma(U\,B_{a,0})$ for any $a$ and $U$ is the same as $\rho(U\,B_{a,0})<1$ for any $a$ and $U$.
\end{proof}

\br \label{ratrmk}
We note that in any case the second condition of Lemma~\ref{Ulem} implies the first. Absent \eqref{ratcond}, however, the reverse may fail, even in the simpler-looking unidirectional case that $m=0$ or $m=n-1$. When the first condition fails to imply the second it is far from implying any form of high-frequency damping estimate. 

Note that in their nonlinear form (not provided here) high-frequency damping estimates must hold for all nearby systems. Hence they are expected to require a generic form of high-frequency spectral gap.
\er

We retain from the high-frequency gap condition only the $a=0$ condition that
\be\label{rat}
\max_{U}\rho(UB)\,<\,1\,.
\ee

\subsection{Linear damping estimate}

Our goal is now to prove a linear damping estimate under the condition 
\be\label{sat}
\inf_{S}\|SBS^{-1}\|\,<\,1\,,
\ee
where again $S$ runs over diagonal, positive matrices. We recall that in any case \eqref{sat} implies \eqref{rat} and that we have proved that, when $n-1\leq5$ i.e. $n\leq 6$, they are equivalent. Obviously our interest for the general linear algebraic equivalent arose the other way around and condition \eqref{sat} appears naturally in the design of a high-frequency damping estimate.

We introduce 
\be \label{lin}
A_0 \d_tw + \d_x(Aw)- Ew=F,\qquad\textrm{on }\tRR 
\ee
together with linearized jump conditions 
\be\label{jump}
\frac{\dD\ y}{\dD t}\,[f_0(\bar W)]-y [R(\bar W)]+ [Aw]=G\,,\qquad\textrm{on }X\Z\,,
\ee
where $F$ and $G$ are given and $y$ is an unkown on $X\Z$.

\bpr\label{ldamping}
Assume the structural assumptions and \eqref{sat}, which holds under the high-frequency spectral stability condition \eqref{rat} when $n\leq 6$. Let $k\in\N$ be such that \eqref{kreg} holds. For some $\theta>0$ and $C$, if $w$ solves \eqref{lin}-\eqref{jump} (for some shift $y$) on a time interval $[0,T_0]$ with $w(0,\cdot)=w_0$ then for all $0\leq t\leq T_0$
\begin{align*}
\|w(\cdot,t)\|_{{H^k(\tRR)}}^2
&\leq C\,\eD^{-\theta t} \|w_0\|_{H^k(\tRR)}^2\\
&\quad+C\int_0^t \eD^{-\theta(t-\tau)}\left(\sum_{\ell=0}^{k-1}\|\d_t^\ell F(\cdot,\tau)\|_{H^{k-\ell}(\tRR)}^2
 +\sum_{\ell=0}^k\|\d_t^\ell G(\cdot,t)\|_{\ell^2(X\Z)}^2\right)\, \dD\tau
\\
&\quad+C\int_0^t \eD^{-\theta(t-\tau)}\big(\|w(\cdot,\tau)\|_{L^2(\tRR)}^2
 + \|y(\cdot,\tau)\|_{\ell^2(X\Z)}^2
 \big)\, \dD\tau.
\end{align*}
\epr

\begin{proof}
To begin with we replace $w$ with $u=T^{-1}w$. Let us first observe that for any $j$
\[
\bp u_+(jX^+,t)\\-u_-(jX^-,t)\ep\,=\,C\,\bp u_+(jX^-,t)\\-u_-(jX^+,t)\ep+\lot
\]
so that by first taking time derivatives and then using the interior equations
\[
\bp (\alpha_+^k\d_x^ku_+)(jX^+,t)\\-(\alpha_-^k\d_x^ku_-)(jX^-,t)\ep\,=\,C\,\bp (\alpha_+^k\d_x^ku_+)(jX^-,t)\\-(\alpha_-^k\d_x^ku_-)(jX^+,t)\ep+\lot\,.
\]
Note that time derivatives of $F$ appear only in the estimates of the latter $\lot$. This is to be combined with
\[
\d_t(\d_x^ku)+\alpha\d_x(\d_x^k u)=-(\gamma+k\,\alpha'+\beta)\,\d_x^ku+\lot\,.
\]

In the above we have already started to consider as lower order terms the contributions of the characteristic modes to the boundary terms. Indeed the sonic part may be treated as in the Saint-Venant case by introducing in the quadratic formula
\[
\langle \Omega_{(s),(C_0)}\d_x^ku_{(s)},\d_x^ku_{(s)}\rangle_{L^2}
\]
with, for $x\in ((j-1)X,jX)$, 
\[
\Omega_{(s),(C_0)}(x)\,:=\,C_0 \exp\left(\int_{(j-1)X}^x \left( 
 (2k-1)\frac{\alpha_{(s)}'(y)-\alpha_{(s)}'(x_s)}{\alpha_{(s)}(y)} + 2\frac{\gamma_{(s)}(y)-\gamma_{(s)}(x_s)}{\alpha_{(s)}(y)}\right)\dD y\right)
\]
where $C_0$ is to be taken sufficiently large. Likewise in the end interior cross terms may be discarded as in the Saint-Venant case by adding to the quadratic functional 
\[
\langle \d_x^ku,K\d_x^{k-1}u\rangle_{L^2}+c_0\|u\|_{L^2}^2
\]
for a well-chosen Kawashima compensator $K$ and a sufficiently large $c_0$.

We may thus focus on transverse modes and do as if $\beta\equiv0$. The question is whether one may create an $\eps$-interior dissipation while deriving dissipative boundary conditions. As in the Saint-Venant case it is sufficient to check that one may obtain strict boundary dissipativity in the $0$-interior dissipation case. Thus for $1\leq \ell\leq n-1$, we introduce for $x\in ((j-1)X,jX)$, 
 \[
\Omega_{\ell,(0),\sigma_\ell}(x)\,:=\,\sigma_\ell\,|\alpha_\ell|^{2k-1}(x)\eD^{\int_{(j-1)X}^x \frac{2\gamma_\ell(y)}{\alpha_\ell(y)}\,\dD y}\,,
\]
to be inserted in 
\[
\langle \Omega_{\ell,(0),\sigma_\ell}\d_x^ku_\ell,\d_x^ku_\ell\rangle_{L^2}
\]
and we note that its main contribution to boundary terms, from which it differs only by $\lot$, is 
\begin{align*}
\frac12\sum_{j\in\Z}&\left((\alpha_\ell \Omega_{\ell,(0),\sigma_\ell} (\d_x^k u_\ell)^2) (jX^-) 
-(\alpha_\ell \Omega_{\ell,(0),\sigma_\ell} (\d_x^k u_\ell)^2) ((j-1)X^+)
\right)\\
&=\frac12\sign(\alpha_\ell)\sum_{j\in\Z}\left(\sigma_\ell\eD^{2\int_0^X \frac{\gamma_\ell}{\alpha_\ell}} ((\alpha_\ell\d_x)^k u_\ell)^2 (jX^-) 
-\sigma_\ell((\alpha_\ell\d_x)^k u_\ell)^2(jX^+)
\right)
\end{align*}
(where we recall that $\alpha_\ell$ has constant sign). This suggest to search for $\sigma_\ell$ under the form
\[
\diag(\sigma_1,\cdots,\sigma_{n-1})
\,=\,S^2\,\diag(\eD^{-2\int_0^X\gamma_+\alpha_+^{-1}},\Id)\,.
\]
with $S$ diagonal and positive. Then summing over $1\leq \ell\leq n-1$ and using boundary equations leave as main boundary contribution
\begin{align*}
\frac12\sum_{j\in\Z}&
\left\langle\bp (\alpha_+^k\d_x^ku_+)(jX^-)\\-(\alpha_-^k\d_x^ku_-)(jX^+)\ep,
\left(S^2-B^TS^2B\right)
\bp (\alpha_+^k\d_x^ku_+)(jX^-)\\-(\alpha_-^k\d_x^ku_-)(jX^+)\ep\right\rangle\\
&=\frac12\sum_{j\in\Z}
\left(\left|S\bp (\alpha_+^k\d_x^ku_+)(jX^-)\\-(\alpha_-^k\d_x^ku_-)(jX^+)\ep\right|^2
-\left|SBS^{-1}\,S\bp (\alpha_+^k\d_x^ku_+)(jX^-)\\-(\alpha_-^k\d_x^ku_-)(jX^+)\ep\right|^2
\right)
\end{align*}
so that it is sufficient to choose $S$ such that $\|SBS^{-1}\|<1$.
\end{proof}

\section{Discussion and open problems}\label{s:disc}

As noted at the beginning of Section \ref{s:general}, the models (RG3) and (RG4) introduced by 
Richard and Gavrilyuk \cite{RG1} as refinements of (SV) are $3\times 3$ and $4\times 4$ models \eqref{blaw},
both satisfying the dimensional requirement $n\leq 6$. The system (RG3) is strictly hyperbolic with characteristics $u-c, u\pm c_s- c$, where $c$ and $c_s$ are respectively wave and sound speeds. Thus, spectrally stable roll waves of (RG3) indeed fall (generically) under the purview of Proposition~\ref{ldamping}. The full model (RG4) is however nonstrictly hyperbolic, possessing an additional characteristic mode with speed $u-c$, hence does not fall directly under the present linearized analysis; however, we suspect it could be treated similarly with additional effort. More generally we expect extension of our analysis to non strictly hyperbolic systems to be both within our reach and of practical relevance. 

An ultimate goal is of course a full nonlinear stability result for roll waves of (SV), (RG3), or (RG4), completing the spectral analysis of \cite{JNRYZ,RYZ3} by establishing that spectral stability implies nonlinear stability. As observed in \cite{JNRYZ,RYZ3}, this is problematic for (SV) and (RG4) due to infinite-dimensional manifolds of nearby stationary solutions, a degenerate situation that apparently requires special handling. For (RG3) on the other hand, there is no such degeneracy, and the spectral picture appears much like that of the viscous case already treated in \cite{JZN,JNRZ13}. The investigation of this problem would be a most interesting direction for future investigation.

A more systematic treatment of high-frequency damping estimates would also be welcome. For our present class of systems the basic open problem is whether in cases where there does exist a gap between high-frequency
stability and existence of the energy estimates developed here one might be able to obtain high-frequency damping estimates by some other techniques. We also mention as an interesting direction the development of a general theory for viscous periodic waves, beyond the Saint-Venant analysis of \cite{RZ}.

An interesting open problem on the linear algebraic side is to find an example of gap for the real-valued case in dimension $n=6$. It would also be of interest, but maybe out of reach, to derive some statistical description of the gap for some ensemble of random matrices so as to elucidate the observations of Remark~\ref{gaprmk}.

\appendix

\section{Counterexample in $\cM_4(\C)$, and extension to $\cM_5(\R)$}\label{s:ceg3} 
In this appendix, we complete the main remaining parts of the proof of Lemma~\ref{linalg_lem} by providing a complex counterexample for dimension $n=4$, and the sharpened bound $n\leq 5$ for the real case.

\subsection{Extension to $\cM_5(\R)$}\label{s:ag}

We begin with the real extension. The overall strategy is the same as in the complex case. The parts checking continuity of the gap in the irreducible case and reducing to the irreducible case are identical. We need to modify the density argument and extend the analysis of the ``simple'' case to some low-multiplicity multiple cases. 

\medskip
\emph{Density argument for full matrices.} We investigate when is dense the set of matrices $B$ such that the infimum of $S\mapsto \|SBS^{-1}\|$ is reached at a $S_*$ such that the largest eigenvalue of $(S_*BS_*^{-1})^*S_*BS_*^{-1}$ has multiplicity at most $m_0$. We do so by examining the orbits under $S\mapsto S B S^{-1}$ of matrices $B$ such that $\rho(B)^2$ is a multiple eigenvalue of $B^*B$ with multiplicity $m>m_0$.  

The space $\cM_n(\R)$ has real dimension $n^2$. On the other hand, as we detail below, by singular-value decomposition $B=LDR^*$, $R$, $L$ orthogonal matrices and $D$ diagonal, real nonnegative, the set of real matrices $B$ for which $B^*B$ has a $m$-tuple largest eigenvalue in $D$ has dimension (at most) $n^2+1-m(m+1)/2$, hence the set of all scalings $SBS^{-1}$ (taking without loss of generality $S_{11}=1$) of such a $B$ has dimension (at most) $n-1$ higher, or (at most) $n^2+ n-m(m+1)$. Indeed the sets of orthogonal matrices $R$ and $L$ have dimensions $n(n-1)/2$ apiece, accounting for $n(n-1)$ degrees of freedom.  Meanwhile, multiplication on the right by the same orthogonal $m\times m$ matrix of the $n\times m$ block of columns associated with the repeated eigenvalue also leaves the  singular-value decomposition unchanged, corresponding to an overcount of $m(m-1)/2$. Subtracting these values from the count, and adding the $(n-m)+1$ parameters corresponding to the entries of $D$, and the $n-1$ parameters corresponding to different scalings $S$ (again, setting the upper left entry to $1$ without loss of generality), we obtain finally (at most) $n(n-1)-m(m-1)/2+(n-m+1)+(n-1)=n^2+n-m(m+1)/2$ as claimed.

We prove below the absence of gap for full real matrices such that the infimum of $S\mapsto \|SBS^{-1}\|$ is reached at a $S_*$ such that the largest eigenvalue of $(S_*BS_*^{-1})^*S_*BS_*^{-1}$ has multiplicity less than $2$. The present density argument extends it to dimensions $n$ such that $n<m(m+1)/2$ when $m\geq 3$, thus to $n\leq 5$ as announced.

\medskip
{\em Double eigenvalue case.}  Before discussing the multiplicity two specifically, we first extend the discussion of the simple case to the multiple case. 

To understand how $S\mapsto \|SBS^{-1}\|$ behaves near the identity matrix, we introduce the following quadratic forms: we denote $Q_j$, $j=1,\cdots,n$, the restriction to $\ker(B^*B-\|B\|^2\Id)$ of quadratic forms on $\C^n$ associated with matrices
\[
2(B^*E_jB-\|B\|^2\,E_j)\,,\quad j=1,\cdots,n\,.
\]
Note that $\sum_{j=1}^nQ_j$ is identically zero. We recall that one may always reduce the problem to the case when the infimum of $S\mapsto \|SBS^{-1}\|$ is taken at $S=\Id$.

It follows from the regularity of spectrum of self-adjoint matrices depending on a one-dimensional parameter that if $\ker(B^*B-\|B\|^2\Id)$ is of dimension $m$ then for any $S$ one may parameterize the $m$ top eigenvalues of $(\Id+\delta S)^{-1}\,B^*\,(\Id+\delta S)^2\,B\,(\Id+\delta S)^{-1}$ in such a way that they depend smoothly on $\delta$ and then their derivative at $\delta=0$ are the eigenvalues of $\sum_{j=1}^n s_j\,Q_j$. This immediately implies the following lemma.

\bl
If $S\mapsto \|SBS^{-1}\|$ admits a local minimum at the identity matrix then no real linear combination of the $Q_j$ is definite.
\el

The treatment of the $m=1$ case already contains the ingredients proving the following lemma.

\bl\label{proptest} 
The $Q_j$ possess a (possibly complex) nonzero common root if and only if there exists $U$ diagonal and complex such that
\[
\rho(U\,B)\,=\,\|B\|\,.
\]
\el

Note that the set of real quadratic forms on a space of dimension $2$ is of dimension $3$ so that $2$ is the maximal number of independent real quadratic forms on a space of dimension $2$ with no real linear combination that is definite. Therefore, to conclude the analysis of the $m=2$ case, that implies the $n\leq 5$ real case of Lemma~\ref{linalg_lem}, we only need to prove the following lemma.

\bl\label{2prop}
Two Hermitian quadratic forms on $\C^2$ either possess a real linear combination that is definite or a nonzero common complex root.
\el

\begin{proof}
Let $q_1$ and $q_2$ be two Hermitian quadratic forms such that there is no real linear combination of $q_1$ and $q_2$ that is strictly definite. Eliminating the trivial case where both forms are zero, changing coordinates and rescaling if necessary, one may assume that $q_1$ and $q_2$ are associated with 
\be\label{Qs}
\bp 1 & 0 \\0 & -a^2\ep, \qquad \bp b & c\\ \bar c & d\ep,
\ee
with $a,b,d$ real and $c$ complex. The condition that $\sigma q_1+q_2$ be indefinite for any $\sigma\in \R$, by Sylvester's criterion, is then that for any $\sigma\in\R$
\[
\det \bp \sigma+b & c \\ \bar c & -a^2\sigma + d\ep\leq 0,
\]
or, equivalently $ -a^2 \sigma^2 + (d-ba^2)\sigma+bd-|c|^2\leq 0$. This is equivalent to 
\be\label{indcrit}
(d+ba^2)^2 \leq 4a^2|c|^2\,.
\ee

On the other hand, the nonzero roots of $q_1$ are, evidently the nonzero multiples of $a$
\be\label{roots}
\bp a \\ \gamma \ep, \qquad |\gamma|=1.
\ee
Substituting \eqref{roots} into $q_2$, we obtain a common root if and only if  $ba^2 + 2\Re (ca \gamma)+ d=0$, or $-2\Re (ca \gamma)= d + ba^2$, which is solvable jointly with $|\gamma|=1$ precisely if $2|ca|\geq |d+ b a^2|$, or \eqref{indcrit}.
\end{proof}

\begin{corollary}\label{real2prop}
Given \emph{any} collection of real quadratic forms $q_k$, $k=1,\cdots,n$ on $\C^2$, either there is a nontrivial real linear combination that is strictly definite, or else there is a nontrivial complex common root.
\end{corollary}

\begin{proof}
As noted above, the dimension of the space of real quadratic forms acting on dimension $m$ is $m(m+1)/2$, thus in the case $m=2$, dimension $3$. Therefore, if there are more than two independent forms, then some linear combination gives the identity, a positive definite matrix. On the other hand, if there are just two independent forms, no linear combination of which is strictly definite, then by Lemma~\ref{2prop} they have a nontrivial common root. But in this case, all other quadratic forms, being linear combinations of these first two, share the same root.
\end{proof}

{\it This completes the proof of Lemma~\ref{linalg_lem} in the positive direction for the real-valued case.}

\subsection{Counterexample in $\cM_4(\C)$}\label{s:xceg}

{\em Quadratic forms.} In Lemma~\ref{2prop} we did not make the restriction to real quadratic forms. However we did use the real character in  Corollary~\ref{real2prop} to reduce the $m=2$ case to the consideration of two quadratic forms. Indeed the set of complex Hermitian quadratic forms on a space of dimension $2$ is of dimension $4$, leaving open the possibility that there could be three quadratic forms on $\C^2$ with no definite real linear combination and no nonzero common complex root. 

This is indeed the case, as can be seen by the explicit counterexample of quadratic forms $q_1$, $q_2$ and $q_3$ associated with matrices
\[
\bp 1 & 0\\ 0 & -1\ep, \quad
\bp 0 & 1\\ 1 & 0 \ep, \quad
\bp 0 & \iD\\ -\iD & 0 \ep.
\]
For the determinant of the matrix associated with $\sum_{k} c_k q_k$ is
\[
\det
\bp c_1 & c_2+ \iD c_3\\ c_2-\iD c_3 & -c_1\ep= -(c_1^2+c_2^2+c_3^2)< 0,
\]
whence every real combination is indefinite. Yet, direct computation shows that the unique common root of $q_1$, $q_2$, and $q_3$ is zero since the equations to satisfy for $x=(x_1,x_2)$ are $|x_1|=|x_2|$, $\Re(\bar{x_1}x_2)=0$, $\Im(\bar{x_1}x_2)=0$. 

This shows that a complex analogous of Corollary~\ref{real2prop}, leaving open the possible open the possibility to design a gap example for elements of $\cM_4(\C)$, as announced in Lemma~\ref{linalg_lem}.

We now design such an example by elaborating on the above quadratic forms. Consider the matrix $B:=LR^* \in \cM_4(\C)$, with $R$ and $L$ elements of $\cM_{4,2}$ having each orthonormal columns $(R_1,R_2)$ and $(L_1,L_2)$, explicitly given by 
\be\label{RY}
R= \frac12\bp 1 & 0\\1&1\\1&\iD\\1& -\iD-1\ep,\qquad
L= \frac12\bp  0&1\\1&-1\\1&-\iD\\ -\iD+1& 1\ep\,.
\ee
By design, $\|B\|=1$, $B^*B$ have eigenvalues $1$ and $0$, both with multiplicity $2$, and the columns of $R$ may be used as a basis of $\ker(B^*B-\|B\|^2\Id)$. When doing so, the quadratic forms $Q_j$, $j=1,\cdots,4$, appearing in the first variation of $S\mapsto \|S B S^{-1}\|$ at $S=\Id$, may be identified with their matrices given as
\[
\bp \,2\,\left(\,\overline{L_{j,k}}L_{j,\ell}-\overline{R_{j,k}}R_{j,\ell}\,\right)\,\ep_{k,\ell=1,2}
\]
and computed to be
\be\label{x2}
Q_1=-\frac14\bp 1 & 0\\ 0 & -1\ep, \quad
Q_2=-\frac14\bp 0 & 2\\ 2 & 0 \ep, \quad
Q_3=-\frac14\bp 0 & 2\iD\\ 2\iD & 0 \ep\,.
\ee
and $Q_4=-(Q_1+Q_2+Q_3)$.

We already know that for any $S$, $\delta\mapsto \|(\Id +\delta S) B (\Id +\delta S)^{-1}\|$ admits a local minimum at $\delta=0$ and by Lemma~\ref{proptest} that
\[
\max_U \rho(UB)<\|B\|\,.
\]

{\em Global minimization.} There only remains to prove that $S\mapsto \|S B S^{-1}\|$ admits a unique global minimizer at $S=\Id$. To confirm this, we introduce
\[
R_S:= S^{-1}R\,, \qquad L_S:= SL\,,
\]
so that $B_S:=SBS^{-1}= L_S R_S^*$. Note that any eigenvector of $B_S^*B_S$ belongs to the range of $R_S$ that $R_S$ is one-to-one and that
\[
R_S^* B_S^* B_S R_S\,=\,(R_S^*R_S)\,(L_S^*L_S)\,(R_S^*R_S)\,.
\] 
Therefore we only need to prove that for any $S$ the largest eigenvalue of $(L_S^*L_S)(R_S^* R_S)$ has value at least $1$. We recall that we already know that all the eigenvalues are real and nonnegative (since they coincide with those of $B_S^*B_S$) thus it is sufficient to prove that the trace equals at least $2$.

Taking without loss of generality $S=\diag(s_1, s_2, s_3, s_4)$, with $s_4=1$, we find after a brief calculation that
\begin{align*}
R_S^*R_S=\bp\sum_{j=1}^4  s_j^{-2}\,\overline{R_{j,k}}\,R_{j,\ell}\ep_{k,\ell}
&=\frac14\bp s_1^{-2}+ s_2^{-2} + s_3^{-2} + 1 & s_2^{-2} -\iD s_3^{-2} + \iD -1\\
s_2^{-2}+ \iD s_3^{-2} -\iD-1 & s_2^{-2} + s_3^{-2} + 2 \ep\\
L_S^*L_S=\bp\sum_{j=1}^4  s_j^2\,\overline{R_{j,k}}\,R_{j,\ell}\ep_{k,\ell}
&=\frac14\bp s_2^2 + s_3^2 + 2 & -s_2^2 + \iD s_3^2 + 1-\iD\\
 -s_2^2 -\iD s_3^2 + 1+\iD& s_1^2 + s_2^2 + s_3^2 + 1 \ep
\end{align*}
giving
\begin{align*}
\Tr&((L_S^*L_S)(R_S^* R_S))
=((L_S^*L_S)\bp1\\0\ep)^*(R_S^*R_S)\bp1\\0\ep
+((L_S^*L_S)\bp0\\1\ep)^*(R_S^*R_S)\bp0\\1\ep\\
&=\frac{1}{16}
\Big(
(s_1^{-2}+ s_2^{-2} + s_3^{-2} + 1)(s_2^2 + s_3^2 + 2)
+(s_1^2 + s_2^2 + s_3^2 + 1)(s_1^{-2}+ s_2^{-2} + s_3^{-2} + 1)\\
&\qquad-2(s_2^2-1)(s_2^{-2}-1)-2(s_3^2-1)(s_3^{-2}-1)
\Big)\,.
\end{align*}
We conclude that indeed $\Tr((L_S^*L_S)(R_S^* R_S))\geq 2$ by using repeatedly that for any $x>0$, $(x-1)(x^{-1}-1)\leq 0$ and $x+x^{-1}\geq 2$.

{\it This verifies global minimality, completing the final step of the proof of Lemma \ref{linalg_lem}.} Note that our proof gives no information on the size of the gap but our numerics show a gap of approximately $10 \%$.

\br\label{gaprmk}
Though the above example demonstrates that a nonzero gap can theoretically exist between high-frequency stability and the condition for existence of our energy estimates, numerical experiments choosing matrices at random suggest that this occurrence is not so frequent, even up to dimension $\C^7$; indeed, we were not able to find by numerical optimization starting from a random initialization existence of matrices with nonzero gap. This suggests that even though strictly stable multiple eigenvalue configurations occur on an open set, the size of this set, and of the resulting gap, may be relatively small, so in practice not an issue. However we must say that we expect gaps to be more frequent when the dimension increases whereas in this direction the optimization problems become harder to solve numerically and our numerics are less conclusive.
\er

\subsection{Towards a counterexample in $\cM_6(\R)$}

Though we did not obtain a complete example of a gap in $\cM_6(\R)$, we do provide some elements in this direction.

First we point out that there are five real quadratic forms on $\C^3$ with no definite real linear combination and no nonzero common complex root. This may be seen by considering the forms associated with
\begin{align*}
\bp 1&0&0\\0&-1&0\\0&0&0\ep&,&
\bp 0&1&0\\1&0&0\\0&0&0\ep&,&
\bp 0&0&1\\0&0&0\\1&0&0\ep&,&
\bp 0&0&0\\0&0&1\\0&1&0\ep&,&
\bp 0&0&0\\0&1&0\\0&0&-1\ep&.&
\end{align*}
Note that $5$ is indeed the maximal number of independent quadratic forms with no real definite real linear combination since the whole space of real quadratic forms has dimension $m(m+1)/2=6$ when $m=3$.

We have looked for $B:=LR^* \in \cM_6(\R)$, with $R$ and $L$ elements of $\cM_{6,3}$ having each orthonormal columns under the form 
\begin{align*}
R&=\bp x_1&y_1&0\\
\rho_2\cos(\theta_2)&\rho_2\sin(\theta_2)&0\\
x_3&y_3&z_3\\
\rho_4\cos(\theta_4)&\rho_4\sin(\theta_4)&z_4\\
\rho_5\cos(\theta_5)\cos(\varphi_5)&\rho_5\sin(\theta_5)\cos(\varphi_5)&\rho_5\sin(\varphi_5)
\ep,\\
L&= \bp y_1&x_1&0\\
\rho_2\cos(\theta_2')&\rho_2\sin(\theta_2')&0\\
-x_3&y_3&z_3\\
\rho_4\cos(\theta_4')&\rho_4\sin(\theta_4')&z_4\\
\rho_5\cos(\theta_5')\cos(\varphi_5')&\rho_5\sin(\theta_5')\cos(\varphi_5')&\rho_5\sin(\varphi_5')
\ep\,,
\end{align*}
with $x_1^2\neq y_1^2$, $\rho_2\neq0$, $x_3\,z_3\neq0$, $\rho_4\,z_4\sin(\theta_4)\neq \rho_4\,z_4\sin(\theta_4')$, $\rho_5\sin(\varphi_5)^2\neq \rho_5\sin(\varphi_5')^2$. The shape and the conditions ensure that the generated quadratic forms share the same span as the above explicit quadratic forms hence share the same properties. 

Note that the shapes of $R$ and $Y$ do not ensure orthonormality \emph{per se} but we have solved numerically for parameters to guarantee orthogonality. We did obtain numerical solutions, corresponding to a $B$ approximately given by
\[
\bp
0.14753503& 0.19982136& 0.00339269& 0.51926021& 0.00847797& 0.21926921\\
0.08321061& 0.13296559& 0.15631294& -0.1954354& -0.04654104& -0.21988828\\
-0.37549381& -0.02645794& 0.01837161& -0.39654006& 0.43179675& 0.49911196\\
-0.3001269& -0.27661858& -0.57464677& 0.35421452& -0.00469564& 0.39153639\\
0.32691344& -0.09925285& 0.51568898& 0.35643262& 0.22186063& 0.25403942
\ep\,.
\]
Our gap numerical tests for this matrix are however non conclusive. 

\section{Energy landscape for $\rho(UB)$}\label{s:UB}

Most of the discussion on the gap linear algebra problem is about the $\inf_S \|S BS^{-1}\|$ problem and the density of cases when a critical point of $S\mapsto\|S BS^{-1}\|$ is also a global minimizer.

It seems interesting to discuss also the energy landscape for the complementary problem $\max_U \rho(UB)$. The situation is dramatically different. There are always plenty of local minimizers that are not global minimizers. Obviously this complicates numerical optimization. Roughly speaking this is consistent with the intuition that the $S$-problem inherits some properties from the convexity of maps $s\mapsto s^{-1}$ and $s\mapsto s$, whereas the $U$-problem has a multiply periodic structure inherited from periodicity of $\theta\mapsto \eD^{\iD \theta}$.

To support this claim, we first provide a trivial example. The construction
\begin{align*}
B&=L\,R^*\,,&
R&=\bp 1\\ -1\ep\,,&
L&=\bp 1\\ 1\ep\,,&
\end{align*}
gives a simple $2\times 2$ example of a matrix for which $U=\Id$ is a local minimum of $U\mapsto \rho(UB)$ since
\[
\rho(\diag(1,\eD^{\iD \theta})B)\,=\,|1-\eD^{-\iD \theta}|
=\sqrt{2(1-\cos(\theta))}\,.
\]

Now we illustrate that starting from dimension $3$ it is easy to produce examples of local maximum that are not global. Pick $\tilde B\in\cM_{n-1}(\C)$ such that
\[
\rho(\tilde B)<\max_U \rho(U \tilde B)\,.
\]
Then for any $r\in(\rho(\tilde B),\max_U \rho(U \tilde B))$
\[
B_r:=\bp r&0_{1,n-1}\\0_{n-1,1}&\tilde B\ep
\] 
is such that $U\mapsto \rho(U B_r)$ admits a local maximum with value $r$ at $U=\Id$ but a larger global maximal value, equal to $\max_U \rho(U \tilde B)$.

\newcommand{\etalchar}[1]{$^{#1}$}

\end{document}